\documentclass{amsart}
\usepackage{amssymb, latexsym}
\theoremstyle{plain}
\newtheorem{theorem}{Theorem}
\newtheorem{corollary}{Corollary}
\newtheorem{lemma}{Lemma}

\theoremstyle{definition}
\newtheorem{definition}{Definition}

\DeclareMathOperator{\Tr}{Tr}

\DeclareMathOperator{\End}{End}
\DeclareMathOperator{\GL}{GL}
\DeclareMathOperator{\Sp}{Sp}
\DeclareMathOperator{\SL}{SL}

\begin{document}
\title[Solvable groups acting on symplectic spreads]
{On the non-existence of certain transitive\\
 actions of solvable groups on symplectic spreads}

\author[Rod Gow]{Rod Gow}
\thanks{School of Mathematics,
University College,
Belfield, Dublin 4,
Ireland,
\emph{E-mail address:} \texttt{rod.gow@ucd.ie}}

\begin{abstract} 
We prove the following result. Let $q$ be a power of an odd prime and let $\Sp(2m,q)$ denote the symplectic group of degree $2m$ over $\mathbb{F}_q$. Then if $q\equiv 1\bmod 4$, no solvable subgroup of $\Sp(2m,q)$ acts transitively on a complete symplectic spread defined on the underlying
vector space of dimension $2m$, unless $m=1$ and $q=5$. The solvable group $\Sp(2,3)$ of order 24 acts transitively on a complete symplectic spread defined on a two-dimensional vector space over $\mathbb{F}_5$. By contrast, when $q\equiv 3\bmod 4$, there is a metacyclic group of order $2m(q^m+1)$ that acts transitively on a complete symplectic spread.

\end{abstract}
\maketitle

 \section{Introduction}
\noindent 
Let $p$ be an odd prime and let $q=p^a$, where $a$ is a positive integer. Let $\mathbb{F}_q$ denote
the finite field of order $q$. 
Let $V$ be a vector space of finite dimension $n$ over $\mathbb{F}_q$. We let $\End(V)$ denote the set of all $\mathbb{F}_q$-linear endomorphisms of $V$ and $\GL(V)$ the group of all invertible elements of $\End(V)$. We also use the notation $\GL(n,q)$ in place of $\GL(V)$, since the isomorphism type of
$\GL(V)$ depends only on $n$ and $q$. 

Suppose now that $n=2m$ is even and that $f:V\times V\to \mathbb{F}_q$ is a non-degenerate alternating bilinear form. 
We say that a subspace $W$ of $V$ is totally isotropic with respect to $f$ if $f(w,w')=0$ for all $w$ and $w'$ in $W$. Given the hypothesis that $f$ is non-degenerate, the largest dimension of a subspace of $V$ that is totally isotropic with respect to $f$ is $m$, and moreover, totally isotropic subspaces of $V$ of dimension $m$ exist. 

We call a collection $\Omega$ of $q^m+1$ $m$-dimensional totally isotropic subspaces $U_i$ of $V$
a \emph{complete symplectic spread} of $V$ if $V$ is the union of the $U_i$
and the $U_i$ intersect trivially in pairs. 

Let us now observe that, since we are working over a finite field, $V$ certainly contains a complete
symplectic spread. This may be proved by the process of field reduction in the following manner. 
Let $X$ be a two-dimensional vector space over the field $\mathbb{F}_{q^m}$ and let 
$F: X\times X \to \mathbb{F}_{q^m}$ be a non-degenerate alternating bilinear form. $X$ has precisely
$q^m+1$ one-dimensional subspaces and it is the union of these subspaces. Since one-dimensional
subspaces are totally isotropic with respect to an alternating bilinear form, and they clearly intersect trivially in pairs, $X$ has a complete symplectic spread.

Let $X'$ denote $X$ considered as vector space of dimension $2m$ over $\mathbb{F}_q$ and let $\Tr:\mathbb{F}_{q^m}\to \mathbb{F}_q$ denote the trace form. We define a bilinear form $g:X' \times X'\to
\mathbb{F}_q$ by
\[
g(u,v)=\Tr(F(u,v))
\]
for all $u$ and $v$ in $X$. It is straightforward to prove that $g$ is alternating and non-degenerate.

Since a one-dimensional $\mathbb{F}_{q^m}$-subspace of $X$ becomes an 
$\mathbb{F}_q$-subspace of $V'$ that is $m$-dimensional and totally isotropic with respect to $g$, it follows that $X'$ possesses a complete symplectic spread with respect to $g$. Finally, we identify $V$ with $X'$.
Since all non-degenerate alternating bilinear forms defined on $V\times V$ are equivalent, it follows that $V$ has a complete symplectic spread with respect to $f$.

The spread that we have constructed has an additional property that we wish to consider 
in this paper. Recall that an isometry of $f$ is an element $\sigma$, say, of $\GL(V)$
that satisfies 
\[
f(\sigma u,\sigma v)=f(u,v)
\]
for all $u$ and $v$ in $V$. The isometries of $f$ form a group, called the isometry group of $f$. Since, as we have already remarked, all non-degenerate alternating bilinear forms defined on $V\times V$ are equivalent, it follows that
the isometry groups of all non-degenerate alternating bilinear forms on $V\times V$ are isomorphic.
Any such isometry group is called the symplectic group of degree $2m$ over $\mathbb{F}_q$ and is denoted
by $\Sp(V)$ or by $\Sp(2m,q)$. 

The isometry group of the alternating bilinear form $F$ defined on $X\times X$ is identical
with the special linear group $\SL(2,q^m)$ of all elements of $\GL(2,q^m)$
of determinant 1. $\SL(2,q^m)$ thus acts as a group
of isometries of the alternating bilinear form $g$ defined on $X'\times X'$, and  may
be considered as a subgroup of $\Sp(2m,q)$. 

Now $\SL(2,q^m)$ acts doubly transitively on the $q^m+1$ one-dimensional subspaces of $X$. It therefore
acts doubly transitively on the complete symplectic spread defined on $X'\times X'$ (with respect to
$g$). By the identification process outlined above, we can say $V$ possesses a complete symplectic spread
whose subspaces are permuted doubly transitively by a subgroup of the isometry group of $f$, this subgroup being isomorphic to $\SL(2,q^m)$. 

Assuming, as before, that $q$ is odd, we note that $\SL(2,q^m)$ is solvable only when $m=1$ and $q=3$. 
We wish to investigate the following question suggested by these observations. Does $V$ possess a complete symplectic spread with respect to some non-degenerate alternating bilinear form $f$ 
whose subspaces are transitively permuted by a solvable subgroup of isometries of $f$? We do not seek a doubly transitive permutation action, as this is much too restrictive. 

It is not difficult to show that the answer is yes when $q\equiv 3\bmod 4$. Indeed, by means of a Galois-theoretic construction, we may show that there is metacyclic group that transitively permutes a complete
symplectic spread on $V$. This construction fails when $q\equiv 1 \bmod 4$ and this failure leads us to ask if there are any transitive actions of solvable groups in this case. The purpose of this paper is to show that with one exception, there are no such transitive actions. Indeed, our main theorem is that if $q\equiv 1\bmod 4$, the only solvable subgroup of $\Sp(2m,q)$ that acts transitively on a complete symplectic spread on $V$ is $\Sp(2,3)=\SL(2,3)$
and this action occurs only when $m=1$ and $q=5$.

\section{Zsigmondy prime divisors}

\noindent  Much of our opening material is well known and has been used in numerous investigations of subgroups of $\GL(n,q)$ and $\Sp(2m,q)$. To keep the document reasonably self-contained, we have provided complete proofs of many of the theorems we need, although we have to appeal to the proofs of a few more complicated results.

We recall that a prime rational integer $r$ is said to be a Zsigmondy prime divisor
of $q^i-1$ if $r$ divides $q^i-1$ but $r$ does not divide $q^j-1$ whenever $1\leq j<i$. The theorem of
Zsigmondy shows that Zsigmondy prime divisors of $q^i-1$ always exist provided $i>2$ and $q$ is odd. See, for example, Theorem 8.3, p.508, of \cite{Hup2}.
It is straightforward to see that a Zsigmondy prime divisor of $q^i-1$ is greater than $i$. For if $r$ is a Zsigmondy prime divisor of $q^i-1$, $q$ has order $i$ modulo $r$ and hence $i$ divides
$r-1$.

Suppose now that $r$ is a Zsigmondy prime divisor of $q^{2m}-1=(q^m-1)(q^m+1)$. Then $r$ divides one of
$q^m-1$ and $q^m+1$, and by the Zsigmondy prime property, $r$ divides $q^m+1$. 

Let us note one case where there is no shortage of Zsigmondy prime divisors of $q^{2m}-1$.

\begin{lemma} \label{power_of_two}
Let $a$ be a positive integer. Then any odd prime divisor of $q^{2^a}+1$ is a Zsigmondy prime divisor
of $q^{2^{a+1}}-1$. 
\end{lemma}

\begin{proof}
Let $r$ be an odd prime divisor of $q^{2^a}+1$ and let $i$ be the order of $q$ modulo $r$. Then certainly $i$ divides $2^{a+1}$, since $r$ divides $q^{2^{a+1}}-1$. Suppose if possible that $i$ is less than
$2^{a+1}$. It follows that $i$ is a divisor of $2^a$ and therefore $q^{2^a}\equiv 1\bmod r$. But since
$r$  divides $q^{2^a}+1$, $q^{2^a}\equiv -1\bmod r$. This is a contradiction, and we deduce that
$q$ has order $2^{a+1}$ modulo $r$. It is immediate from this fact that $r$ is a Zsigmondy divisor
of $q^{2^{a+1}}-1$. 
\end{proof}

\begin{definition}
Let $\sigma$ be an element of $\End(V)$ and let
\[
\mathcal{C}(\sigma)=\{ \tau\in \End(V): \tau\sigma=\sigma\tau\}
\]
and 
\[
\mathcal{P}(\sigma)=\{ a(\sigma): a(x)\in \mathbb{F}_q[x]\}.
\]
\end{definition}

Thus $\mathcal{C}(\sigma)$ consists of all elements of $\End(V)$ that commute with $\sigma$. It is an
$\mathbb{F}_q$-subalgebra of $\End(V)$. $\mathcal{P}(\sigma)$
consists of all polynomials in $\sigma$ and it is a subalgebra of $\mathcal{C}(\sigma)$.

The following is a well known result of linear algebra, whose proof we omit.

\begin{lemma} \label{dimension_of_algebra}
Let $\sigma$ be an element of $\End(V)$ and let $d$ be the degree of its minimal polynomial. Then
$\mathcal{P}(\sigma)$ has dimension $d$.
\end{lemma}

We proceed to the proof of another well known result.

\begin{lemma} \label{irreducible_action}
Let $r$ be a Zsigmondy prime divisor of $q^{n}-1$ and let $\sigma$ be a non-identity
element of $\GL(V)$ of order a power of $r$. Then the cyclic subgroup generated by $\sigma$ acts irreducibly on $V$
and the minimal polynomial of $\sigma$ is irreducible of degree $n$.
\end{lemma}

\begin{proof}
Let $H$ denote the cyclic subgroup of $\GL(V)$ generated by $\sigma$. Suppose that
$H$ acts reducibly on $V$. Then, since $r$ is relatively prime to $q$, the Maschke-Schur theorem, \cite{Hup1}, Satz 17.7, p.123, implies
that $V=V_1\oplus V_2$, where $V_1$ and $V_2$ are proper $H$-invariant subspaces of $V$. $H$ must act faithfully on at least one of $V_1$, $V_2$, and we can assume the notation is chosen so that
we know $H$ acts faithfully on $V_1$. 

Let $t=\dim V_1$. Then $H$ may be considered to be a subgroup of $\GL(V_1)$
 and hence $|H|$ divides $|\GL(V_1)|$. Since we have
\[
|\GL(V_1)|=q^{t(t-1)/2}\prod _{i=1}^t (q^i-1),
\]
it follows that $r$ divides $q^i-1$, where $i\leq t<n$. This contradicts the assumption that
$r$ is a Zsigmondy prime divisor of $q^{n}-1$, and it follows that $H$ acts irreducibly.

Let $s(x)$ be the minimal polynomial of $\sigma$. For any non-zero vector $u$ in $V$, the subspace $U$ of $V$ spanned by the elements $\sigma^i u$ is invariant under $H$ and hence equals $V$, by irreducibility.
But $U$ has dimension at most $\deg s(x)$, and thus $n\leq \deg s(x)$. Since $\deg s(x)$ is certainly at most $n$, the equality follows.
\end{proof}

\begin{lemma} \label{equality_of_algebras}
Let $r$ be a Zsigmondy prime divisor of $q^{n}-1$ and let $\sigma$ be a non-identity
element of $\GL(V)$ of order a power of $r$. Then the algebra $\mathcal{C}(\sigma)$ is isomorphic to the finite
field $\mathbb{F}_{q^n}$. Furthermore, we have $\mathcal{C}(\sigma)=\mathcal{P}(\sigma)$. 
\end{lemma}

\begin{proof}
Let $u$ be an arbitrary non-zero element of $V$. We have shown in the proof of Lemma \ref{irreducible_action} that any element in $V$ may be expressed as $a(\sigma)u$
for some polynomial $a(x)$ in $\mathbb{F}_q[x]$. 
Suppose now that $\tau\in \mathcal{C}(\sigma)$. Then $\tau u=
a(\sigma) u$ by our observation above, for some polynomial $a(x)$. We claim that $\tau=a(\sigma)$. 

For let $v$ be any vector in $V$. We may write $v=b(\sigma)u$, for some polynomial $b(x)$. We obtain
\[
\tau v=\tau b(\sigma)u=b(\sigma) \tau u=b(\sigma)a(\sigma) u=a(\sigma)b(\sigma) u=a(\sigma) v.
\]
Thus $\tau$ and $a(\sigma)$ agree on all elements of $V$ and are thus equal. This shows that
$\mathcal{C}(\sigma)=\mathcal{P}(\sigma)$.

Since $\sigma$ acts irreducibly on $V$, Schur's Lemma, \cite{Hup1}, Lemma 10.5, p.56, implies that each non-zero element of $\mathcal{P}(\sigma)$ is invertible.
Thus $\mathcal{P}(\sigma)$ is a finite commutative integral domain, and hence a finite field. Moreover, its dimension is the degree of the minimal polynomial of $\sigma$, which we have shown in Lemma \ref{irreducible_action}
is $n$. Thus, $\mathcal{C}(\sigma)$ is isomorphic to $\mathbb{F}_{q^n}$, since there
is up to isomorphism a unique finite field of dimension $n$ over $\mathbb{F}_{q}$.
\end{proof}

\begin{lemma} \label{cyclic}
Let $r$ be a Zsigmondy prime divisor of $q^{n}-1$ and let $R$ be a non-trivial $r$-subgroup of $\GL(V)$. Then $R$ is cyclic. Thus, a Sylow $r$-subgroup of $\GL(V)$ is cyclic.
\end{lemma}

\begin{proof}
Let $\sigma$ be an element of order $r$ in the centre of $R$. Then $\sigma$ acts irreducibly on $V$ by Lemma
\ref{irreducible_action} and $\mathcal{C}(\sigma)=\mathcal{P}(\sigma)$ is a finite field
by Lemma \ref{equality_of_algebras}. Since $R$ commutes with $\sigma$, $R$ is contained in $\mathcal{C}(\sigma)$. Therefore, $R$ is a subgroup of the multiplicative group of a finite field and hence cyclic.
\end{proof}

Suppose now that $\dim V=2m$ is even and that $r$ is a Zsigmondy prime divisor of $q^{2m}-1$.
Then $r$ divides $|\Sp(2m,q)|$ and it follows from Lemma \ref{cyclic} that any non-identity $r$-subgroup
of $\Sp(2m,q)$ is cyclic and acts irreducibly on $V$. We need to know the structure of the normalizer 
in $\Sp(2m,q)$ of any such subgroup. We begin by finding the structure of its centralizer.

\begin{lemma} \label{symplectic_centralizer}
Let $r$ be a Zsigmondy prime divisor of $q^{2m}-1$ and let $R$ be a non-identity $r$-subgroup
of $\Sp(2m,q)$. Let $C_S(R)$ denote the centralizer of $R$ in $\Sp(2m,q)$. Then
$C_S(R)$ is cyclic of order $q^m+1$. 
\end{lemma}

\begin{proof}
We know from Lemma \ref{cyclic} that $R$ is cyclic and we
let $\sigma$ be a generator of $R$. Lemma \ref{equality_of_algebras} shows that
$\mathcal{P}(\sigma)=\mathcal{C}(\sigma)$ is isomorphic to $\mathbb{F}_{q^{2m}}$. 

Let $f$ be a non-degenerate $R$-invariant alternating bilinear form defined on $V\times V$.
Given $\tau$ in $\mathcal{C}(\sigma)$, we define its adjoint $\tau*$ by the formula
\[
f(\tau^* u,v)=f(u,\tau v)
\]
for all $u$ and $v$ in $V$. It is clear that $\tau^*\in \mathcal{C}(\sigma)$ and the mapping $\tau\to \tau^*$ is an $\mathbb{F}_q$-automorphism of $\mathcal{C}(\sigma)$ whose order is one or two. Since $\sigma\in \mathcal{C}(\sigma)$ and $\sigma^*=\sigma^{-1}\neq \sigma$, the automorphism of $\mathcal{C}(\sigma)$ is not the identity and hence has order two. 

Now the only
$\mathbb{F}_q$-automorphism of $\mathbb{F}_{q^{2m}}$ of order two is given by $\alpha\to \alpha^{q^m}$.
Furthermore, the elements $\tau$ in $C_S(R)$ are characterized as those that satisfy $\tau^*=\tau^{-1}$.
Thus, $C_S(R)$ consists of those elements $\tau$ that satisfy $\tau^{q^m+1}=1$ and  is therefore a cyclic group of order $q^m+1$, since the invertible elements of $\mathcal{C}(\sigma)$ form a cyclic group of order
$q^{2m}-1$. 
\end{proof}

Next, we try to determine the structure of the normalizer of $R$ in $\Sp(2m,q)$. The initial information
we obtain is somewhat imprecise and we will have to improve it as we proceed with our investigations.

\begin{lemma} \label{symplectic_normalizer}
Let $r$ be a Zsigmondy prime divisor of $q^{2m}-1$ and let $R$ be a non-identity $r$-subgroup
of $\Sp(2m,q)$. Let $N_S(R)$ denote the normalizer of $R$ in $\Sp(2m,q)$. Then
$N_S(R)/C_S(R)$ is cyclic of order dividing $2m$. 
\end{lemma}

\begin{proof}
Let $\tau$ be an element of $N_S(R)$. Then it is clear that for each element $a(\sigma)$ of $\mathcal{C}(\sigma)$, 
\[
\tau a(\sigma)\tau^{-1} =a(\tau\sigma\tau^{-1})
\]
is also an element of  $\mathcal{C}(\sigma)$. Thus if we define $\alpha_\tau:  \mathcal{C}(\sigma)\to  \mathcal{C}(\sigma)$ by
\[
\alpha_\tau (a(\sigma))=\tau a(\sigma)\tau^{-1},
\]
 we see that $\alpha_\tau$ is an automorphism
of $\mathcal{C}(\sigma)$. Moreover, $\tau\to \alpha_\tau$ is a homomorphism from $N_S(R)$ into
the automorphism group of $\mathcal{C}(\sigma)$ whose kernel is $C_S(R)$. Since by Lemma \ref{equality_of_algebras},
$\mathcal{C}(\sigma)$ is isomorphic to the finite field $\mathbb{F}_{q^{2m}}$, whose automorphism group
over $\mathbb{F}_{q}$ is cyclic of order $2m$, we deduce that $N_S(R)/C_S(R)$ is isomorphic to a subgroup
of a cyclic group of order $2m$.
\end{proof}

To be more precise about the structure of $N_S(R)$, we will use some properties of Galois fields to construct an explicit model of the subgroup.

Let $V'$ denote $\mathbb{F}_{q^{2m}}$ considered as a $2m$-dimensional vector space over $\mathbb{F}_q$. Let $\omega$ be a generator of the multiplicative group of $\mathbb{F}_{q^{2m}}$ and let $\epsilon=\omega^{(q^m+1)/2}$. Then we have 
$\epsilon^{q^m}=-\epsilon$, as we may verify. Let $\Tr$ denote the trace form from
$\mathbb{F}_{q^{2m}}$ to $\mathbb{F}_{q}$.

We define a bilinear form $f$ on $V'\times V'$ by setting 
\[
f(x,y)=\Tr(\epsilon xy^{q^m})
\]
for all $x$ and $y$ in $\mathbb{F}_{q^{2m}}$. 

\begin{lemma} \label{nondegeneracy}
The bilinear form $f$ defined above is alternating and non-degenerate.
\end{lemma}

\begin{proof}
We first show that $f$ is alternating. Now the trace form is invariant under Galois automorphisms and thus
\[
f(x,y)=\Tr(\epsilon xy^{q^m})=\Tr((\epsilon xy^{q^m})^{q^m})=-\Tr(\epsilon x^{q^m}y)=-f(y,x)
\]
for all $x$ and $y$ in $\mathbb{F}_{q^{2m}}$. This implies that $f$ is alternating.

Suppose now that $y$ is in the radical of $f$. Then we have 
\[
\Tr(\epsilon x y^{q^m})=0
\]
for all $x$. Since the trace form is not identically zero, this implies that $y=0$, and thus $f$ is indeed non-degenerate.
\end{proof}

We now define two further elements $\lambda$ and $\mu$ of $\mathbb{F}_{q^{2m}}$, as follows.
We set
\[
\lambda=\omega^{(q-1)/2},   \quad \mu=\omega^{(q^m-1)}
\]
It is then elementary to verify that $\mu$ has multiplicative order $q^m+1$.

Next, we define elements $\pi$ and $\rho$ of $\GL(V')$, as follows. We set
\[
\pi x=\lambda x^q, \quad \rho x=\mu x
\]
for all $x$ in $V'=\mathbb{F}_{q^{2m}}$.

\begin{lemma} \label{two_isometries}
The elements $\pi$ and $\rho$ defined above are isometries of $f$.
\end{lemma}

\begin{proof}
For any $x$ and $y$ in $\mathbb{F}_{q^{2m}}$, we have 
\[
f(\pi x,\pi y)=\Tr(\epsilon \lambda x^q(\lambda y^q)^{q^m})=\Tr(\epsilon \lambda^{q^m+1}x^qy^{q^{m+1}}).
\]
When we recall that $\Tr$ is Galois invariant, we see that 
\[
f(x,y)=\Tr(\epsilon xy^{q^m})=\Tr(\epsilon^q x^q y^{q^{m+1}}).
\]
Thus, to prove that $\pi$ is an isometry, we must show that
\[
\epsilon^q=\epsilon \lambda^{q^m+1}.
\]
Now, given the definitions of $\epsilon$ and $\lambda$, we have
\[
\epsilon^q=\omega^{q(q^m+1)/2}, \quad \epsilon\lambda^{q^m+1}=\omega^{(q^m+1)/2}\omega^{(q^m+1)(q-1)/2}
\]
and these equations imply the desired equality. Hence, $\pi$ is  an isometry.

As far as $\rho$ is concerned, we have
\[
f(\rho x,\rho y)=\Tr(\epsilon \mu x \mu^{q^m}y^{q^m})=f(x,y),
\]
since $\mu^{q^m+1}=1$, and this shows that $\rho$ is also an isometry.
\end{proof}

It is straightforward to prove that $\pi$ has order $4m$ and $\pi^{2m}=-I$. Similarly, $\rho$ has order $q^m+1$, with $\rho^{(q^m+1)/2}=-I$.

\begin{lemma} \label{relations_between_pi_and_rho}
The isometries $\pi$ and $\rho$ satisfy the relation
\[
\pi\rho=\rho^q\pi.
\]
\end{lemma}

\begin{proof}
Let $x$ be an element of $\mathbb{F}_{q^{2m}}$. Then we have
\[
\pi (\rho x)=\lambda (\mu x)^q=\lambda\mu^q x^q.
\]
Similarly, 
\[
\rho^q(\pi x)=\mu^q \lambda x^q
\]
and this clearly provides the desired equality.
\end{proof}

Let $G$ be the subgroup of the isometry group of $f$ generated by $\pi$ and $\rho$, and let
$H$ be the subgroup generated by $\rho$. Lemma \ref{relations_between_pi_and_rho} implies that $\pi$ normalizes $H$, and indeed it induces an automorphism of order $2m$ of $H$, with $\pi^{2m}=-I$ centralizing $H$.

We now describe the structure of $G$ in the circumstances that are relevant to this paper.

\begin{lemma} \label{structure_of_G}
Suppose that $q\equiv 1 \bmod 4$ or that $m$ is even. Let $A$ and $B$ be the subgroups of $G$ generated by
$\rho^2$ and $\pi$, respectively. Then $A\cap B=I$ and $G$ is the semi-direct product of the normal cyclic subgroup $A$ of odd order $(q^m+1)/2$ with the cyclic subgroup $B$ of order $4m$. $G$ is thus metacyclic of order $2m(q^m+1)$. A Sylow $2$-subgroup of $G$ is cyclic. $G$ contains a unique involution,
$-I$. The normalizer of a subgroup of order $4$ in $G$ has order $4m$.
\end{lemma}

\begin{proof}
Let us set $\theta=\rho^2$. Suppose that $\theta^b=\pi^c$ is in $A\cap B$. Then for all $x$ in
$\mathbb{F}_{q^{2m}}$ we have
\[
\theta^b x=\mu^{2b} x=\pi^c x=\lambda^{(q^c-1)/(q-1)}x^{q^c}.
\]
We take $x=1$ and deduce that $\mu^{2b}=\lambda^{(q^c-1)/(q-1)}$. Hence $x^{q^c}=x$ for all $x$ in
$\mathbb{F}_{q^{2m}}$. This implies that $2m$ divides $c$, say $c=2md$, for some integer $d$. 

Now we have $\pi^{2m}=-I$ and thus $\pi^c=(-I)^d$. However, as $\theta$ has odd order under the given hypothesis, an equation of the form $\theta^b=(-I)^d$ implies that $d$ is even and $\theta^b=I=\pi^c$.
This proves that $A\cap B=I$.

Let $T$ be a Sylow 2-subgroup of $G$. Since $A$ has odd order, $A\cap T =I$ and it follows that $T$ is isomorphic to a subgroup of $G/A\cong B$. Thus $T$ is cyclic of order equal to the 2-part of $4m$.

Sylow's theorem implies that all involutions in $G$ are conjugate to the unique involution in $T$. Since
$G$ contains the central involution $-I$, whose only $G$-conjugate is itself, $-I$ is the unique involution in $G$. Similarly, any subgroup of order 4 in $G$ is conjugate to the subgroup $Y$, say, generated by $\pi^m$. Let $N_G(Y)$ denote the normalizer of $Y$ in $G$.
Clearly, $Y$ is centralized by $B$ and hence $B\leq N_G(Y)$.

Suppose that $N_G(Y)$ is strictly larger than $B$. Then $N_G(Y)$ must intersect
$A$ non-trivially, and since $A$ is normal in $G$, $\pi^m$  centralizes some non-identity element of
$A$. We have seen that $\pi$ satisfies $\pi \rho=\rho^q \pi$ and thus similarly $\pi \theta =\theta^q \pi$. Consequently, 
\[
\pi^m \theta =\theta^{q^m} \pi^m.
\]
But $\theta^{q^m}=\theta^{-1}$ and it follows that $\pi^m$ inverts all elements of $A$ when acting by conjugation. Thus, as $A$ has odd order, the only element of $A$ centralized by $\pi^m$ is the identity.
This proves that $N_G(Y)=B$, as required
\end{proof}

We may clearly consider the group $G$ described above to be a subgroup of $\Sp(2m,q)$. It contains a normal cyclic subgroup $A$ of odd order $(q^m+1)/2$ when $q\equiv 1 \bmod 4$ or when $m$ is even. Let $r$ be a Zsigmondy prime divisor of $q^{2m}-1$ and let $C$ be a subgroup of $G$ of order $r^t$, where $t\geq 1$. $C$ is a characteristic subgroup of $A$ and is hence normal in $G$. It follows that $G\leq N_S(C)$, where $S$ denotes $\Sp(2m,q)$. 
Now Lemma \ref{symplectic_normalizer} implies that $N_S(C)$ has order dividing $2m(q^m+1)$.
Since $|G|=2m(q^m+1)$, we deduce that $G=N_S(C)$. Finally, Sylow's theorem implies that any subgroup of $\Sp(2m,q)$ of order $r^t$ is conjugate to $C$, and hence has normalizer isomorphic to $G$. We summarize this argument below.

\begin{corollary} \label{normalizer_of_r_subgroup}
Let $r$ be a Zsigmondy prime divisor of $q^{2m}-1$ and let $C$ be a non-identity $r$-subgroup of $\Sp(2m,q)$. Then if $q\equiv 1\bmod 4$ or if $m$ is even, the normalizer of $C$ in $\Sp(2m,q)$ has order $2m(q^m+1)$ and is isomorphic to the group $G$ described in Lemma \ref{structure_of_G}.
\end{corollary}

The following consequence of this result is important for establishing the structure of certain subgroups of $\Sp(2m,q)$.

\begin{corollary} \label{normal_r_subgroup}
Let $r$ be a Zsigmondy prime divisor of $q^{2m}-1$ and let $H$ be a subgroup of $\Sp(2m,q)$ whose order is divisible by $r$. Suppose that $H$ contains a non-identity normal
$r$-subgroup. Then if $q\equiv 1\bmod 4$ or if $m$ is even, $H$ is isomorphic to  a subgroup
of the group $G$ described in Lemma \ref{structure_of_G}. If $|H|$ is even, a Sylow $2$-subgroup of $H$ is cyclic
and $H$ contains a unique involution, $-I$. If $4$ divides $|H|$, the normalizer of a subgroup of order $4$ in $H$ has order dividing $4m$.
\end{corollary}

It is perhaps worth pointing out that if $q\equiv 3\bmod 4$ and $m$ is odd, a Sylow 2-subgroup of the group described in Corollary \ref{normal_r_subgroup} may be generalized quaternion and the normalizer of a subgroup of order 4 may have order greater than $4m$.
This can already be seen when $m=1$. As the cyclicity and normalizer statements are key points in our subsequent argument, it is important to observe how exceptions can occur.

In the proof that follows, given a finite group $G$ and a prime divisor $r$ of $|G|$, we let $O_r(G)$ denote the largest normal subgroup of $G$ whose order is a power of $r$.

\begin{theorem} \label{solvable_group}
Let $G$ be a solvable subgroup of $\GL(n,q)$. Suppose that $|G|$ is divisible by  a Zsigmondy prime divisor $r$, say, of $q^{n}-1$. Let $R$ be a Sylow $r$-subgroup of $G$. Then
$R$ is normal in $G$, unless possibly $n=2^a$ and $r=2^a+1$ is a Fermat prime.
\end{theorem}

\begin{proof}
Recall that as a Zsigmondy prime divisor of $q^{n}-1$, $r$ is greater than $n$. The theorem then follows from a result of Ito, \cite{Hup2}, Theorem 8.5, p.509,  if $p$, the characteristic of $\mathbb{F}_q$, does not divide $|G|$. 

Suppose next that $p$ divides $|G|$. Let $M$ be a Hall $\{p, r\}$-subgroup of $G$, which we may assume contains $R$. $R$ acts irreducibly on $V$, and hence so also does $M$. Moreover, Lemma \ref{cyclic} shows that $R$ is cyclic.

Since $M$ acts irreducibly on $V$, $O_p(M)=1$ by \cite{Hup1}, Satz 5.17,  p.485. Thus $O_r(M)>1$. The Hall-Higman lemma, \cite{Is}, Lemma 14.22, p.256,  implies that
$O_r(M)$ contains its centralizer in $M$. Since $O_r(M)$ is contained in $R$, and hence is centralized by
$R$, $O_r(M)=R$. Thus, $R$ is normal in $M$.

Let $\pi$ be the set of prime divisors of $|G|$ with the exception of $p$. Let $H$ be a Hall
$\pi$-subgroup of $G$. We may assume that $R$ is contained in $H$. The argument above implies that $R$ is normal in $H$ unless possibly $n=2^a$ and $r=2^a+1$ a Fermat prime. Then, if the exceptions above do not occur, $R$ is normal in both $H$ and $M$ and hence also in $HM=G$. 
\end{proof}

We proceed to use Theorem \ref{solvable_group} to obtain our main working theorem concerning solvable subgroups of $\Sp(2m,q)$  whose order is divisible by $q^m+1$. 

\begin{theorem} \label{solvable_subgroup_of_the_symplectic_group}
Let $G$ be a solvable subgroup of $\Sp(2m,q)$, where $q$ is odd and $m\geq 2$.
Suppose also that $q\equiv 1\bmod 4$ or that $m$ is even. Suppose furthermore that $q^m+1$ divides $|G|$. Then there is at least one Zsigmondy prime divisor $r$ of $q^{2m}-1$ such
that a Sylow $r$-subgroup of $G$ is normal in $G$, unless possibly $q=3$ and $m=2$.
\end{theorem}

\begin{proof}
We first observe that $q^m+1\equiv 3 \bmod 4$ under the hypothesis that $q\equiv 1 \bmod 4$
or $m$ is even. Let $r$ be a Zsigmondy prime divisor of $q^{2m}-1$ and let $R$ be a Sylow $r$-subgroup of $G$. Then $R$ is cyclic and non-trivial, and $R$ is normal in $G$, unless possibly $m=2^b$ and $r=2^{b+1}+1$ is a Fermat prime, by Theorem \ref{solvable_group}. 

Suppose then that $R$ is not normal in $G$. Then we take $m=2^b$ and $r=2^{b+1}+1$. Lemma 
\ref{power_of_two} shows that any odd prime divisor of $q^{2^b}+1$ is a Zsigmondy prime divisor of $q^{2^{b+1}}-1$, and $r$ is such a prime. Suppose if possible that $r_1$ is a different  odd prime divisor of $q^{2^b}+1$. Then a Sylow $r_1$-subgroup of $G$ is also cyclic, non-trivial, and normal in $G$ by Theorem \ref{solvable_group}, since $r_1\neq r$. Thus we are finished unless the only odd prime divisor of $q^{2^b}+1$ is $r$. 

We may now assume that $q^{2^b}+1=2r^t$, where $t$ is a positive integer. Let $M$ be a
Sylow 2-complement of $G$ that contains $R$. Theorem 8.15, p.509,  of \cite{Hup2}, together with the argument used to prove Theorem \ref{solvable_group}, shows that $R$ is normal in $M$. Let
$H$ be a Hall $\{2,r\}$-subgroup of $G$ containing $R$. We intend to show that $R$ is normal in $H$, and this will imply that $R$ is normal in $G$, with a single possible exception when
$q=3$ and $m=2$.

For the sake of this proof, we may therefore assume that $G=H$ and thus the only prime divisors of $|G|$ are 2 and $r$. We claim that $O_r(G)=1$. For suppose that $T=O_r(G)\neq 1$. Now Lemma \ref{symplectic_centralizer} implies that the centralizer of $T$ in $\Sp(2m,q)$ has order $q^m+1$. Since the 2-part of $q^m+1$ is 2 under
our hypotheses, it follows that the only elements of order a power of 2 in 
$\Sp(2m,q)$ that centralize $T$ are $\pm I$. 

The centralizer $C_G(T)$ of $T$ in $G$ is normal in $G$, since $T$ is normal, and it contains $R$, since $T$ is contained in $R$ and $R$ is cyclic. In view of our observation above on the elements of 2-power order in $\Sp(2m,q)$ that centralize $T$, we see that 
$C_G(T)$ either equals $R$ or it is the direct product of $R$ with the cyclic group generated by $-I$. In either case, $R$ is characteristic in $C_G(T)$ and hence normal in $G$, contrary to hypothesis. Thus, $O_r(G)=1$, as claimed.

Since $G$ is solvable and $O_r(G)=1$, it follows that $O_2(G)\neq 1$. Therefore, $G$ contains a non-trivial normal abelian 2-subgroup, $Z$, say. Let $R_1$ be the subgroup of order
$r$ in $R$ and let $K=ZR_1$. $K$ acts irreducibly on $V$, since $R_1$ does so. Clifford's theorem applied to $K$ implies that if $V_Z$ denotes $V$ considered as a $Z$-module, we have
\[
V_Z=V_1\oplus \cdots \oplus V_k,
\]
where each $V_i$ is a direct sum of isomorphic irreducible $Z$-modules, and the $V_i$ are permuted transitively by $R_1$. Since $|R_1|=r$, it follows that $k=1$ or $k=r$. However, if $k=r$, it follows that $\dim V=r\dim V_1\geq r$ and this is impossible, since $\dim V<r$.
Thus $k=1$, and we deduce that $V_Z$ is a direct sum of isomorphic irreducible $Z$-modules.

 This implies that $Z$ has a faithful irreducible module (since $V$ is necessarily faithful for $Z$) and is thus cyclic. See, for example, Theorem 3.2, p.267, of \cite{Hup2}. Now $R_1$ acts by conjugation on $Z$ and induces an automorphism of order one or $r$ on $Z$. The automorphism must then be trivial, since the automorphism group of a cyclic 2-group is itself a 2-group. This implies that $R_1$ centralizes $Z$. The earlier part of this argument shows that, since $R_1$ acts irreducibly,
 the only non-trivial element of 2-power order in $\Sp(2m,q)$ that centralizes $R_1$ is $-I$, and thus $Z$ is the cyclic group generated by $-I$. 
 
 Now since $O_r(G)$ is trivial, the Hall-Higman lemma  previously quoted implies that $O_2(G)$ contains its centralizer in $G$ and hence $O_2(G)\neq Z$. Let then $N/Z$ be a minimal normal subgroup
 of $G/Z$ contained in $O_2(G)/Z$. $N/Z$ is an elementary abelian 2-group and thus
 $N'\leq Z$, where $N'$ denotes the commutator subgroup of $N$. $N$ is not abelian, since we have just shown that the only normal abelian subgroup of 2-power order in $G$ is $Z$, and so we must have $N'=Z$. Furthermore, let $Z(N)$ denote the centre of $N$. $Z(N)$ contains $Z$ and is a characteristic abelian subgroup of $N$. Therefore, $Z(N)$ is a normal abelian 2-subgroup of $G$ and hence equals $Z$, by our earlier arguments. Thus, $N'=Z(N)$ and it follows that $N$ is an extra-special 2-group. 
 
 We have 
 $|N/Z|=2^{2k}$ for some positive integer $k$ and $N$ has a unique irreducible faithful representation, which has degree $2^k$. See, for example, the proof of Theorem 8.5, p.511, of \cite{Hup2}. This representation is defined over $\mathbb{F}_p$, and it follows by Clifford's theorem that $V_N$ is a direct sum of copies of this unique faithful irreducible representation. Thus $2^k$ divides $\dim V=2^{b+1}$ and we deduce that
 $k\leq b+1$.
 
 The group $R$ acts as a group of automorphisms of $N/Z$ in the following way: given $g\in R$ and coset $xZ$ of $N/Z$, $g$ sends $xZ$ to $gxg^{-1}Z$. We claim that no non-identity element $g$ of $R$ fixes a coset $xZ$ different from $Z$. For suppose that we have
 \[
 gxg^{-1}Z=xZ.
 \]
 Then $gxg^{-1}=x(-I)^d$, where $d=0$ or $d=1$. Then $g^2 x g^{-2}=x$ and we see that
 $g^2$ commutes with $x$. Since $g$ has odd order, $g$ commutes with $x$. But $x$ has order a power of 2 and our previous arguments now show that 
 if $g\neq 1$, then $x=\pm I$ and hence $xZ=Z$. This establishes our claim.
 
 We deduce that the $2^{2k}-1$ cosets of $N/Z$ different from $Z$ are the disjoint union of
  $R$-orbits each of size $|R|$. Thus, $|R|$ divides $2^{2k}-1=(2^k-1)(2^k+1)$. This certainly implies that $|R|\leq 2^k+1$ and thus $|R|\leq 2^{b+1}+1$. 
  
  Recall that we have already established that $|R|=(q^{2^b}+1)/2$ and our inequality above
  implies that
  \[
  (q^{2^b}+1)/2\leq 2^{b+1}+1.
  \]
The trivial estimate $q>2$ leads to the inequality
\[
2^{2^b}< 2^{b+2}+1
\]
and hence $2^b\leq b+2$. It follows that $b\leq 2$.

The estimate $q\geq 3$ shows that we cannot have $b=2$, and the possibility that $b=0$ is excluded by our assumption that $\dim V\geq 4$. This leaves $b=1$ and we have 
\[
q^2+1\leq 10.
\]
Thus $q=3$ and $\dim V=4$, a case that cannot be excluded.

\end{proof}

We remark that $\Sp(4,3)$ contains a subgroup of order 160 which is an extension of an extra-special group of order 32 by a group of order 5, and the existence of this group implies that our estimate in the proof is tight.

\section{Group action on spread}

\noindent Suppose that $q\equiv 1 \bmod 4$ and let $\sigma$ be an element
 in $\Sp(2m,q)$ that satisfies $\sigma^2=-I$. Let $\omega$ be an element in $\mathbb{F}_q$ that satisfies $\omega^2=-1$. Given a $\sigma$-invariant subspace $U$ of $V$, we set
\[
U_\omega(\sigma)=\{ u\in U: \sigma u=\omega u\}, \quad U_{\omega^{-1}}(\sigma)=\{ u\in U: \sigma u=\omega^{-1} u\}. 
\]

$U_\omega(\sigma)$ and $U_{\omega^{-1}}(\sigma)$ are the $\omega$- and $\omega^{-1}$-eigenspaces of $\sigma$ acting on $U$, respectively. For the sake of brevity, we shall denote them simply by $U_\omega$ and $U_{\omega^{-1}}$ when no confusion concerning their dependence on $\sigma$ can arise.
These subspaces have the following well known property.

\begin{lemma} \label{direct_sum_of_isotropics}
Suppose that $q\equiv 1 \bmod 4$. Let $\sigma$ be an element
 in $\Sp(2m,q)$ that satisfies $\sigma^2=-I$ and let $U$ be a $\sigma$-invariant subspace of $V$. Then we have 
\[
U=U_\omega\oplus U_{\omega^{-1}}
\]
and the subspaces 
$U_\omega$ and $U_{\omega^{-1}}$ are both totally isotropic.
\end{lemma}

\begin{proof}
Let $u$ be any element of $U$. We can write
\[
2u=u-\omega \sigma u+u+\omega \sigma u.
\]
Now we have 
\[
\sigma(u -\omega \sigma u)=\sigma u+\omega u=\omega(u -\omega \sigma u)
\]
and likewise 
\[
\sigma(u +\omega \sigma u)=\sigma u-\omega u=\omega^{-1}(u +\omega \sigma u)
\]
These two equations imply that $U=U_\omega\oplus U_{\omega^{-1}}$. 

Consider now arbitrary elements $u$ and $w$ in $U_\omega$. We have then
\[
f(u,w)=f(\sigma u, \sigma w)=f(\omega u,\omega w)=-f(u,w)
\]
and this shows that $f(u,w)=0$. Thus $U_\omega$ is totally isotropic, and a similar proof yields the same conclusion for $U_{\omega^{-1}}$. 
\end{proof}

\begin{corollary} \label{both_have_dimension_m}
Suppose that $q\equiv 1 \bmod 4$. Let $\sigma$ be an element
 in $\Sp(2m,q)$ that satisfies $\sigma^2=-I$. Then the subspaces 
$V_\omega$ and $V_{\omega^{-1}}$ both have dimension $m$.
\end{corollary}

\begin{proof}
Since $V_\omega$ and $V_{\omega^{-1}}$ are both totally isotropic, they both have dimension
at most $m$. However, since 
\[
2m=\dim V=\dim V_\omega+ \dim V_{\omega^{-1}},
\]
it is clear that we must then have $\dim V_\omega=\dim V_{\omega^{-1}}=m$.
\end{proof}

\begin{corollary} \label{three_subspaces}
Suppose that $q\equiv 1 \bmod 4$, Let $\sigma$ be an element
 in $\Sp(2m,q)$ that satisfies $\sigma^2=-I$. Let $X$, $Y$ and $Z$ be three different $\sigma$-invariant  totally isotropic $m$-dimensional subspaces of $V$ that satisfy $X\cap Y=Y\cap Z=Z\cap X=0$. Then $m$ is even 
 and the six subspaces
 $X_\omega$, $X_{\omega^{-1}}$, $Y_\omega$, $Y_{\omega^{-1}}$, 
 $Z_\omega$, and $Z_{\omega^{-1}}$ all have dimension $m/2$.
 \end{corollary}
 
 \begin{proof}
 Let $r=\dim  X_\omega$ and $s=X_{\omega^{-1}}$, where $r+s=m$. We may then find a basis 
 $x_1$, \dots, $x_m$ of $X$ such that $\sigma x_i =\lambda_i x_i$, for $1\leq i\leq m$, where $\lambda_i=\omega$ for $1\leq i\leq r$, and $\lambda_i=\omega^{-1}$ for $r+1\leq i\leq m$. 
 
 Now, since $X$ and $Y$ have dimension $m$ and intersect trivially, $V=X\oplus Y$. Moreover, as $X$ and $Y$ are both totally isotropic and the underlying alternating bilinear form is non-degenerate, there is a basis $y_1$, \dots, $y_m$ of $Y$ that satisfies
 \[
 f(x_i,y_j)=\delta_{ij}, \quad 1\leq i, j\leq m,
 \]
 where $\delta_{ij}$ denotes the usual Kronecker delta. It follows that as $Y$ is $\sigma$-invariant, $\sigma y_i=\mu_i y_i$, where $\mu_i=\omega^{-1}$ for  $1\leq i\leq r$, and $\mu_i=\omega$ for $r+1\leq i\leq m$. It is straightforward then to see that 
 $y_1$, \dots, $y_r$ are a basis of $Y_{\omega^{-1}}$, and $y_{r+1}$, \dots, $y_m$ are a basis of $Y_{\omega}$, and therefore
 \[
 \dim X_\omega=\dim Y_{\omega^{-1}}, \quad \dim X_{\omega^{-1}}=\dim Y_{\omega}.
 \]
 
 When we apply the same argument to $X$ and $Z$, we obtain
  \[
 \dim X_\omega=\dim Z_{\omega^{-1}}, \quad \dim X_{\omega^{-1}}=\dim Z_{\omega}.
 \]
 In the same manner, using $Y$ and $Z$, we obtain
  \[
 \dim Y_\omega=\dim Z_{\omega^{-1}}, \quad \dim Y_{\omega^{-1}}=\dim Z_{\omega}.
 \]
 
 These six equations imply that all six subspaces have the same dimension, and then since
 we have $m=\dim X_\omega+\dim X_{\omega^{-1}}$, we deduce that $m$ is even and $m/2$ is the dimension in common.
 \end{proof}
 
 \begin{corollary} \label{action_of_element_of_order_four}
Suppose that $q\equiv 1 \bmod 4$. Let $\sigma$ be an element
 in $\Sp(2m,q)$ that satisfies $\sigma^2=-I$. Suppose that $\sigma$ permutes the elements of a complete symplectic spread, $\Omega$, say, of $V$. Let $\Omega_\sigma$ denote the subset of elements of $\Omega$ fixed by $\sigma$. Then either $|\Omega_\sigma|=2$ or $m$ is even and $|\Omega_\sigma|=q^{m/2}+1$. In the former case,
 the subspaces $V_\omega$ and $V_{\omega^{-1}}$ are members of $\Omega_\sigma$, 
 in the latter they are not members of $\Omega$.
 \end{corollary}
 
 \begin{proof}
 Let $v$ be a non-zero element of $V_\omega$. Then $v$ is in some member of $\Omega$, $X$, say, and then $\sigma v=\omega v$ is in $\sigma X$. Consequently, as
 $\sigma X$ is a subspace, we have $v\in X\cap \sigma X$. Since $\sigma X$ is also a member of $\Omega$, and different members of $\Omega$ have trivial intersection with each other, we must have $\sigma X=X$. A similar
 argument applies to elements of $V_{\omega^{-1}}$.
 
 Now there are precisely $2(q^m-1)$ non-zero vectors that are eigenvectors of $\sigma$ and thus the argument above implies that  $|\Omega_\sigma|\geq 2$.
 Let us suppose that $\sigma$ fixes exactly two subspaces in $\Omega$, $X$ and $Y$, say. It follows that 
 \[
 X \cup Y=V_\omega\cup V_{\omega^{-1}}, \quad X_\omega\cup Y_\omega=V_\omega.
 \]
 
 Let $r=\dim X_\omega$. Then $m-r=\dim V_{\omega^{-1}}$. The proof of Corollary \ref{three_subspaces} shows that $\dim Y_\omega=\dim V_{\omega^{-1}}=m-r$. Thus, since 
 $X_\omega\cup Y_\omega=V_\omega$, we must have
 \[
 q^r-1 +q^{m-r}-1=q^m-1.
 \]
 This is clearly only possible if $r=0$ or $m-r=0$. Now $r=0$ implies that $X= V_{\omega^{-1}}$ and $Y=X_\omega$, whereas $m-r=0$ implies that $X=V_\omega$ and 
$Y= V_{\omega^{-1}}$. Thus the subspaces $V_\omega$ and $V_{\omega^{-1}}$ are members of $\Omega$ in this case.

Suppose instead that $|\Omega_\sigma|>2$. Corollary \ref{three_subspaces} implies that $m$ is even and if $Z$ is any member of $\Omega_\sigma$, $\dim Z_\omega=m/2$. Let $Z^1$, \dots, $Z^t$ be the different elements of $\Omega_\sigma$. Clearly, from the argument above, $Z^i_\omega$ is a subspace of
$V_\omega$ of dimension $m/2$. Conversely, each element of $V_\omega$ is in some $Z^i_\omega$, as we have shown above. Thus, since $Z^i\cap Z^j=0$ if $i\neq j$, the subspaces
$Z^i_\omega$ form a complete (ordinary) spread of subspaces of dimension $m/2$ of $V_\omega$
and thus $t=q^{m/2}+1$.
 \end{proof}
 
 \begin{theorem} \label{no_transitive_action_in_solvable_case}
 Let $G$ be a solvable subgroup of $\Sp(2m,q)$. Suppose that $m\geq 2$ and $q\equiv 1 \bmod 4$. Then $G$ does not act transitively on any complete symplectic spread of $V$.
 \end{theorem}

\begin{proof}
Suppose by way of contradiction that $G$ acts transitively on some complete symplectic spread, $\Omega$, say, of $V$. It follows that $q^m+1$ divides $|G|$, and thus $G$ has even order. Theorem \ref{solvable_subgroup_of_the_symplectic_group} implies that there is at least one Zsigmondy prime divisor $r$, say, of $q^{2m}-1$ such that a Sylow $r$-subgroup $R$, say, of $G$ is normal in $G$. 

It follows from the normality of $R$ that $G$ is a subgroup of $N_S(R)$, where $N_S(R)$ denotes the normalizer of $R$ in $\Sp(2m,q)$. Lemma \ref{structure_of_G} now implies that
a Sylow 2-subgroup of $N_S(R)$ is cyclic and $N_S(R)$ contains a unique involution, $-I$.
These two conclusions also apply to $G$, as it is a subgroup of $N_S(R)$. 

Let $Z$ be the central subgroup generated by $-I$. $Z$ clearly fixes each subspace in $\Omega$ and is thus contained in the stabilizer subgroup, $\Gamma$, say, of any given
subspace in $\Omega$. Since $|G:\Gamma|=q^m+1$, we see that 4 divides $|G|$. This implies that $G$ contains elements of order 4, since $G$ has a cyclic subgroup of order divisible by 4. 

Let $\sigma$ be an element of order 4 in $G$. Then as $\sigma^2$ has order 2, $\sigma^2=-I$.
It follows from Corollary \ref{action_of_element_of_order_four} that either $|\Omega_\sigma|=2$ or $m$ is even and $|\Omega_\sigma|=q^{m/2}+1$.

Suppose that the first possibility occurs. Then $\Omega_\sigma$ consists of the two eigenspaces $V_\omega(\sigma)=V_\omega$ and $V_{\omega^{-1}}(\sigma)$, where $\omega$ has order 4 in $\mathbb{F}_q$. But $V_\omega$ is clearly fixed by the centralizer $C_G(\sigma)$ of $\sigma$ in $G$ by virtue of its defining property as an eigenspace. $C_G(\sigma)$ contains a Sylow 2-subgroup of $G$, since the Sylow 2-subgroups are cyclic. This implies that if $\Gamma$ is the stabilizer of $V_\omega$, $|G:\Gamma|$ is odd and hence cannot equal $q^m+1$, which is even. The first possibility considered
has thus led to a contradiction.

It must be the case that $m$ is even and $|\Omega_\sigma|=q^{m/2}+1$. Let $Y$ be the subgroup of $G$ generated by $\sigma$ and let $N_G(Y)$ denote the normalizer of $Y$ in $G$. It is clear that
$N_G(Y)$ acts on $\Omega_\sigma$ and we will show that it acts transitively. 

For let $U_1$, \dots, $U_t$ be the subspaces in $\Omega_\sigma$, where $t=q^{m/2}+1$, and let $\Gamma$ be the stabilizer of $U_1$. Let $U_i=x_i U_1$, $1\leq i\leq t$, where the $x_i$ 
are left coset representatives of $\Gamma$ in $G$. Since $Y$ fixes $U_i=x_iU_1$, $x_i^{-1}Yx_i$ is contained in $\Gamma$. Now the Sylow 2-subgroup of $\Gamma$ is cyclic and it therefore follows from Sylow's theorem that any two subgroups of order 4 in $\Gamma$ are conjugate in $\Gamma$. Thus we must have 
\[
x_i^{-1}Yx_i=y_i^{-1}Yy_i,
\]
where $y_i\in \Gamma$ for $1\leq i\leq t$. It follows that $x_iy_i^{-1}\in N_G(Y)$. Thus if we set $x_iy_i^{-1}=z_i$, then $z_i\in N_G(Y)$ and
\[
U_i=x_iU_1=z_iy_iU_1=z_iU_1.
\]
This proves that the action of $N_G(Y)$ on $\Omega_\sigma$ is transitive. Furthermore, since $Y$ acts trivially on $\Omega_\sigma$, the size of $\Omega_\sigma$ is at most $|N_G(Y):Y|=
|N_G(Y)|/4$.

Referring to Lemma \ref{structure_of_G}, it follows that $N_G(Y)$ is cyclic of order dividing $4m$, and thus $|N_G(Y):Y|\leq m$. We have therefore established that
\[
|\Omega_\sigma|=q^{m/2}+1\leq m
\]
and this inequality is clearly impossible. Thus the second possibility also cannot occur, and we deduce that $G$ does not act transitively on any complete symplectic spread.
\end{proof}

We turn to consideration of the single example of the exceptional behaviour of a solvable group acting on a symplectic spread.

\begin{theorem} \label{exceptional_behaviour_when_p=5}
 Let $G$ be a solvable subgroup of $\Sp(2,q)$, where $q\equiv 1 \bmod 4$. Then $G$ does not act transitively on any complete symplectic spread on $V$ unless $q=5$. The group $\SL(2,3)$ of order $24$ acts transitively on a complete symplectic spread defined on a vector space
 of dimension two over $\mathbb{F}_5$.
 \end{theorem}

\begin{proof}
Since $q\equiv 1\bmod 4$, $q+1$ is not a power of 2, and any odd prime divisor, $r$, say,
of $q+1$ is a Zsigmondy divisor of $q^2-1$. Then, unless $r=3$, a Sylow $r$-subgroup of $G$ is normal and we may finish the proof as we did in Theorem \ref{no_transitive_action_in_solvable_case}.

There remains the case that the only odd prime divisor of $q+1$ is 3, so that $q+1=2\cdot3^t$. Let $R$ be a Sylow 3-subgroup of $G$, where $|R|=3^t$. The proof of Theorem
\ref{solvable_subgroup_of_the_symplectic_group} shows that $|R|=3$ and hence $q=5$. Moreover, 24 divides $|G|$ in this case.

It is easy to check that the only proper subgroup of $\Sp(2,5)$ of order divisible by 24 is $\SL(2,3)$ and that this group does act transitively on a complete symplectic spread over $\mathbb{F}_5$.
\end{proof}

We conclude our arguments by resolving one further unsettled case. 

\begin{theorem} \label{no_exceptional_behaviour}
Let $G$ be a solvable subgroup of $\Sp(2m,q)$, where $q\equiv 3\bmod 4$ and $m$ is even. Suppose that $G$ acts transitively on a complete symplectic spread of $G$. Then there is at least one Zsigmondy prime divisor $r$ of $q^{2m}-1$ such
that a Sylow $r$-subgroup of $G$ is normal in $G$.
\end{theorem}

\begin{proof}
In view of Theorem \ref{solvable_subgroup_of_the_symplectic_group}, we have only to consider the case that $q=3$, $m=2$ and $r=5$. We suppose that a Sylow 5-subgroup of $G$ is not normal in $G$ and then derive a contradiction.

Our proof of Theorem \ref{solvable_subgroup_of_the_symplectic_group} shows that in these circumstances $G$ contains a normal extra-special 2-subgroup of order 32. It follows that $|G|=2^a\cdot 3^b\cdot 5$, where $a\geq 5$. We will now show that $b=0$ and thus $G$ is a $\{2,5\}$-group.

Let $H$ be a Hall  $\{3,5\}$-subgroup of $G$. Since a Sylow 5-subgroup of $G$ acts irreducibly on $V$, it follows that $H$ also acts irreducibly on $V$. Thus, by earlier arguments, $O_3(H)=1$, and we must have $O_5(H)>1$. This implies that a Sylow 5-subgroup, $R$, say, of $H$ is normal in $H$ and thus $H\leq N_S(R)$, where $S=\Sp(4,3)$. Corollary \ref{normalizer_of_r_subgroup} implies that $|N_S(R)|=40$ and we deduce that $|H|=5$. This establishes that $|G|=2^a\cdot 5$, where $a\geq 5$.

Let $U$ be a 2-dimensional totally isotropic subspace in $\Omega$ and let $\Gamma$ be its stabilizer subgroup in $G$. We have then $|G:\Gamma|=10$, since $G$ acts transitively on
$\Omega$. Thus $|\Gamma|=2^{a-1}\geq 16$. 

Since $\Gamma$ is a 2-group, it fixes an even number of elements of $\Omega$. Thus, there is some other subspace $W$, say, contained in $\Omega$ that is also fixed by $\Gamma$. We have then
$V=U\oplus W$, since $U\cap W=0$. Now $\Gamma$ acts on $U$, and we claim that this action is faithful. 

For suppose that some element $\sigma$ of $\Gamma$ acts trivially on $U$. Since $U$ and $W$ are both totally isotropic, there exist bases $\{u_1, u_2\}$ of $U$ and $\{w_1, w_2\}$ of $W$ such that 
\[
f(u_1,w_1)=f(u_2, w_2)=1, \quad f(u_1,w_2)=f(u_2, w_1)=0.
\]
If we now use the fact that $\sigma$ is an isometry of $f$ that fixes $u_1$ and $u_2$, and maps $W$ into itself, we see that $\sigma$ also fixes $w_1$ and $w_2$, and therefore $\sigma$ is the identity. This proves our claim.

It follows that $\Gamma$ is isomorphic to a subgroup of $\GL(2,3)$ and thus has order dividing $|\GL(2,3)|=48$. As we have already shown that $|\Gamma|\geq 16$, we deduce that $|\Gamma|=16$. Therefore, $\Gamma$ is a Sylow 2-subgroup 
of $\GL(2,3)$ and hence a semi-dihedral group of exponent 8. On the other hand, we now know that a Sylow 2-subgroup of $G$ has order 32, since the 2-part of $|G|$ is $2|\Gamma|$, and is therefore extra-special of exponent 4.
This contradicts the fact that $\Gamma$ has exponent 8. We conclude that a Sylow 5-subgroup of $G$ is normal and the proof is complete.
\end{proof}

It is worth mentioning that when $q\equiv 3\bmod 4$, the group $G$ described before Lemma \ref{structure_of_G} does act transitively on a complete symplectic spread defined on $V$.

\end{document}